\documentclass{amsart}
\usepackage{amsfonts,amssymb,amsmath,amsthm}
\usepackage{url}
\usepackage{enumerate}

\usepackage{amsbsy,hyperref,cite,epstopdf}
\usepackage[all,v2,cmtip]{xy}

\newtheorem{thrm}{Theorem}[section]
\newtheorem{lem}[thrm]{Lemma}
\newtheorem{prop}[thrm]{Proposition}
\newtheorem{cor}[thrm]{Corollary}
\theoremstyle{definition}
\newtheorem{definition}[thrm]{Definition}
\newtheorem{remark}[thrm]{Remark}
\newtheorem{conjecture}[thrm]{Conjecture}
\newtheorem{assumption}[thrm]{Assumption}
\numberwithin{equation}{section}

\newcommand{\de}[4]{\delta([\nu^{#1}\rho_{#2}, \nu^{#3}\rho_{#4}])}

\author{Yeansu Kim}
\address{
Department of Mathematics Education\\
Chonnam National University\\
Gwangju city, Korea}
\email{ykim@jnu.ac.kr}
\thanks{}

\begin{document}

\title[Langlands-Shahidi $L$-functions for $GSpin$ groups]{Langlands-Shahidi $L$-functions for $GSpin$ groups and the generic Arthur packet conjecture}

\begin{abstract}
We prove that $L$-functions from Langlands-Shahidi method in the case of $GSpin$ groups over a non-Archimedean local field of characteristic zero are Artin $L$-functions through the local Langlands correspondence. It has an application on the proof of a weak version of the generic Arthur packet conjecture. Furthermore, we study and describe a local $L$-packet that contains a generic member in the case of $GSpin$ groups. Using this description of a local $L$-packet, we strengthen a weak version of the generic Arthur packet conjecture in those cases (local version of the generalized Ramanujan conjecture).
\end{abstract}
\maketitle

\section{Introduction}
\label{Sec1}

The construction and understanding of local $L$-packets, the disjoint subsets of irreducible admissible representations of a quasi-split reductive group $\textbf{G}$ over a $p$-adic local field, has been one of the main subjects in the Langlands program. The generic representations (representations with a Whittaker model) play an important role in parametrizing tempered $L$-packets. For example, it has been conjectured that, when $F$ is local, every tempered $L$-packet contains a generic representation. This conjecture is called the tempered $L$-packet conjecture (see \cite{Shahidi:1990a} for more details). In this paper, we consider the converse to the tempered $L$-packet conjecture. The converse to this conjecture is not true in general. However, if we consider a local $L$-packet which corresponds to an Arthur parameter, the converse is expected to be true. More precisely, let $\psi$ be an Arthur parameter of $\textbf{G}$ over a $p$-adic local field $F$ as defined in Section \ref{Sec5}. There is a Langlands parameter $\phi_{\psi}$ which corresponds to the Arthur parameter $\psi$. We consider an $L$-packet $\Pi(\phi_{\psi})$ that is attached to this Langlands parameter $\phi_{\psi}$. The following conjecture is called a weak version of the generic Arthur packet conjecture \cite{Shahidi:2011}:

\begin{conjecture}\label{Conjecture A} Assume that there exists an $L$-packet $\Pi(\phi_{\psi})$ that corresponds to the Langlands parameter $\phi_{\psi}$, where $\phi_{\psi}$ corresponds to the Arthur parameter $\psi$. If $\Pi(\phi_{\psi})$ has a generic member, then the Langlands parameter $\phi_{\psi}$ is tempered.
\end{conjecture}

The main result of our paper, the equality of $L$-functions from Langlands-Shahidi method for $GSpin$ groups and the corresponding Artin $L$-functions through the local Langlands correspondence (Theorem \ref{Theorem B}), implies Conjecture \ref{Conjecture A} for $GSpin$ groups \cite[Theorem 5.1]{Shahidi:2011}. To explain our results more precisely, let $\textbf{G}_n:=\textbf{GSpin}_{2n}$ or $\textbf{GSpin}_{2n+1}$ (resp. $\textbf{GL}_m$) denote the general spin group of semisimple rank $n$ (resp. general linear group of degree $m$) over $F$ and let $G_n$ (resp. $GL_m$) denote the group of $F$-points of $\textbf{G}_n$ (resp. $\textbf{GL}_m$). Briefly, in the case of $GSpin$ groups, there are two types of $L$-functions that Shahidi defined in a series of papers \cite{Shahidi:1978, Shahidi:1981, Shahidi:1988, Shahidi:1990a, Shahidi:2010} (Langlands-Shahidi method). The first $L$-function is a Rankin product $L$-function for $GL \times GSpin$, denoted $L(s, \pi' \times \pi)$, where $\pi'$ (resp. $\pi$) denotes an irreducible admissible generic representation of $GL_m$ (resp. $G_n$) and the second $L$-function is either a twisted symmetric square $L$-function or a twisted exterior square $L$-function depending on whether $G_n$ is an even $GSpin$ group or an odd $GSpin$ group. (See Section \ref{Sec3.2} for more details). 

Henniart \cite{Henniart:2010} has proved that both the twisted symmetric square and twisted exterior square $L$-functions are Artin $L$-functions (note that it is also recently proved that the twisted exterior and symmetric square $\gamma$-factors are Artin factors in \cite{Cogdell_Shahidi_Tsai:2014, Ganapathy_Lomeli:2014}). Therefore, it remains to consider the Rankin product $L$-functions for $GL \times GSpin$. We prove the following result (Theorem \ref{E:main:general}):

\begin{thrm}\label{Theorem B} Let $\pi$ be an irreducible admissible generic representation of $G_n$. Then, there exists an irreducible admissible representation $\Pi$ of $GL$ such that for any irreducible admissible generic representation $\pi'$ of $GL_m$ we have
$$L(s, \pi' \times \pi) = L(s, \pi' \times \Pi) \textit{ and } \gamma(s, \pi' \times \pi, \psi_F) = \gamma(s, \pi' \times \Pi, \psi_F).$$
\end{thrm}

Here, the $\gamma$-functions, i.e., $\gamma(s, \pi' \times \pi, \psi_F)$ and $\gamma(s, \pi' \times \Pi, \psi_F)$ are complex functions defined in \cite[Theorem 3.5]{Shahidi:1990a}. The $L$-function in the right hand side, i.e., $L(s, \pi' \times \Pi)$ is a Rankin-Selberg $L$-function \cite{Jacquet_PS_Shalika:1983}. The Rankin-Selberg $L$-functions are Artin $L$-functions due to the local Langlands correspondence for general linear groups \cite{Harris_Taylor:2001, Henniart:2000}. Therefore, Theorem \ref{Theorem B} implies that the Rankin product $L$-functions for $GL \times GSpin$ are also Artin $L$-functions. 

As an application of our main theorem (Theorem \ref{Theorem B} and \cite[Theorem 5.1]{Shahidi:2011}), we prove the weak version of the generic Arthur packet conjecture for $GSpin$ groups:
\begin{thrm}\label{Theorem A}
Conjecture \ref{Conjecture A} is true in the case of $GSpin$ groups.
\end{thrm}


\begin{remark}\label{Heiermann and Kim}
In \cite{Heiermann_Kim:2017}, Heiermann and the author proved Theorem \ref{Theorem B} with different approach following Heiemrann's construction of Langlands parameters \cite{Heiermann:2004, Heiermann:2006, Heiermann:2006b}. In the present paper, we give another proof of Theorem \ref{Theorem B} using classification results Section \ref{Sec3.1} or \cite{YKim:2015, YKim:2016}. However, note that the classification results behave an important role when we prove several properties in local Langlands functoriality from $GSpin$ groups to $GL$ groups and $L$-packets, for example, temperedness is preserved through the local Langlands functoriality from GSpin groups to $GL$ groups (Proposition \ref{tempered} and Lemma \ref{TT}). Furthermore, as an application of those properties, we define and study generic $L$-packets (Definition \ref{Def:generic L-packet}) and we strengthen Conjecture \ref{Conjecture A} in the case of $GSpin$ groups (The strong version of the generic Arthur packet conjecture for $GSpin$ groups).
\end{remark}


The second purpose of our paper is to strengthen the generic Arthur packet conjecture. There are two points to strengthen the conjecture. First part is to remove the assumption on the existence of $L$-packets, i.e., define $L$-packets. Second part is to get stronger results of the conjecture, i.e., prove strong version of the generic Arthur packet conjecture. 

First, We define and study $L$-packets that contain a generic member, called generic $L$-packets, in the case of $GSpin$ groups to make Conjecture \ref{Conjecture A} be less conjectural. More precisely, in \cite{Heiermann_Kim:2017}, Heiermann and the author constructed Langlands parameters that correspond to irreducible admissible generic representation of $GSpin$ groups. Next project is to describe a set of all irreducible admissible representations that correspond to the Langlands parameters, generic $L$-packets.
First step of this project is to describe all generic representations in generic $L$-packets in terms of $L$-functions from Langlands-Shahidi method (Theorem \ref{generic reps in $L$-packet}). Following this result, we define $L$-packets as the set of irreducible admissible representations that share the same local factors (Definition \ref{Def:generic L-packet}) with one assumption on the existence of $L$-functions for non-generic representation 

Second, as an application of the description of the generic $L$-packets, we strengthen a weak version of the generic Arthur packet conjecture. The following is called a strong version of the generic Arthur packet conjecture:

\begin{conjecture}\label{Conjecture A'} 
Let $\psi$ and $\phi_{\psi}$ be as in Conjecture \ref{Conjecture A} and let $\Pi(\phi_{\psi})$ be the $L$-packet defined in Definition \ref{Def:generic L-packet}. If $\Pi(\phi_{\psi})$ has a generic member, then $\Pi(\phi_{\psi})$ is a tempered $L$-packet. (Here, tempered $L$-packet means that all members in an $L$-packet $\Pi(\phi_{\psi})$ are tempered).
\end{conjecture}

\begin{remark}
The Conjecture \ref{Conjecture A'} can be considered to be a local version of the generalized Ramanujan conjecture and the conjecture itself is related to the generalzied Ramanujan conjecture. (See \cite{Shahidi:2011} for more details).
\end{remark}

We prove the following:

\begin{thrm}\label{Theorem A'}
 The Conjecture \ref{Conjecture A'} is true in the case of $GSpin$ groups.
\end{thrm}

Let us briefly explain the main idea of the proof. Theorem \ref{Theorem A} implies that the corresponding Langlands parameters are tempered. Furthermore, it is well known that, in the case of general linear groups, $L$-packets for $GL$ are tempered if and only if the corresponding $L$-parameters are tempered \cite{Harris_Taylor:2001, Henniart:2000}. More precisely, let $\pi \in \Pi(\phi_{\psi})$ be an irreducible admissible generic representation of $GSpin$ groups and let $\Pi$ be its functorial lift to $GL$. Then we have

$$\xymatrixcolsep{7pc}\xymatrix{
\pi: generic \ar[d]_{Theorem \ A} &\Pi:tempered \\
\phi_{\psi}: tempered \ar[r]^{\textit{\cite[Theorem \ 5.1]{Heiermann_Kim:2017}}} &\phi_{\Pi}:tempered \ar[u]_{\textit{\cite{Harris_Taylor:2001, Henniart:2000}}}}$$

The first step of the proof is to show that $\pi$ in the above diagram is tempered. In other words, we show that the property `being tempered' is compatible with the local functoriality from $GSpin$ groups to $GL$ groups (Lemma \ref{TT}):

\begin{lem}\label{Lemma A'} Let $\pi$ be an irreducible admissible generic representation of $GSpin$ groups. Then, $\pi$ is tempered if and only if $\Pi$ is tempered.
\end{lem}
\begin{remark}
Lemma \ref{Lemma A'} does not depend on our definition of $L$-packets (Definition \ref{Def:generic L-packet}) and therefore it is 
unconditional.
\end{remark}

Lemma \ref{Lemma A'} implies that $L$-packet $\Pi(\phi_{\psi})$ contains a tempered representation $\pi$. Then, the final step of the proof is to show that all other members in $\Pi(\phi_{\psi})$ are tempered as well. This follows by Lemma \ref{tempered $L$-packet}.

\begin{remark}
The techniques used from Theorem \ref{HeK} to Lemma \ref{Lemma A'}, i.e., Section \ref{Sec5} can be applied to the cases of classical groups and we leave this for future work (Remark \ref{Classical groups}).
\end{remark}

The paper is organized as follows. In Section \ref{Sec2}, we introduce the notation. In Section \ref{Sec3}, we briefly recall the results on the classification of strongly positive representations. We also recall the $L$-functions from Langlands-Shahidi method in the case of $GSpin$ groups. In Section \ref{Sec4}, we show that the Rankin product $L$-functions for $GL \times GSpin$ are equal to Artin $L$-functions. More precisely, we first state the result on the equality of $L$-functions in the supercuspidal case (Theorem \ref{sc}). In Section \ref{Sec4.2}, we formulate the classificataion result in the generic case (the embedding of discrete series generic representations of $GSpin$ groups; Theorem \ref{classification:generic}). We then use this classification result to show the equality of $L$-functions in the case of discrete series representations (Proposition \ref{ds}). After that, we use the Langlands classification and the properties of tempered representations to show the equality of $L$-functions in general (Proposition \ref{tempered}, Theorem \ref{E:main} and Theorem \ref{E:main:general}). In Section \ref{Sec5}, we describe and study $L$-packets that contain generic representations. We also introduce the generic Arthur packet conjecture. This conjecture is one application of our main result, i.e., the equality of $L$-functions and we prove it for $GSpin$ groups (the weak version of the generic Arthur packet conjecture: Theorem \ref{ConjG:GSpin}). Furthermore, we strengthen the conjecture and prove it in the case of $GSpin$ groups (the strong version of the generic Arthur packet conjecture: Theorem \ref{StrongConjG}).

\section{Notation}
\label{Sec2}

 Let $F$ be a non-Archimedean local field of characteristic zero and let $\textbf{G}_n:=\textbf{GSpin}_{2n}$ or $\textbf{GSpin}_{2n+1}$ (resp. $\textbf{GL}_m$) denote the general spin group of semisimple rank $n$ (resp. general linear group of degree $m$) over $F$ and let $G_n$ (resp. $GL_m$) denote the group of $F$-points of $\textbf{G}_n$ (resp. $\textbf{GL}_m$). 
Let $\textbf{s}=(n_1, n_2, \ldots, n_k)$ be an ordered partition of some $n'$ such that $n' \leq n$. Let $P_{\textbf{s}}=M_{\textbf{s}}N_{\textbf{s}}$ denote a standard $F$-parabolic subgroup of $G_n$ that corresponds to the partition ${\textbf{s}}$. The Levi factor $M_{\textbf{s}}$ is isomorphic to $GL_{n_1} \times GL_{n_2} \times \cdots \times GL_{n_k} \times G_{n-n'} \ (\text{see \cite{Asgari:2002}} ).$
We denote the induced representation $Ind_{P_{\bf s}}^{G_n}(\rho_1 \otimes \cdots \otimes \rho_k \otimes \tau)$ by
$$ \rho_1 \times \cdots \times \rho_k \rtimes \tau$$
\noindent where each $\rho_i$ (resp. $\tau$) is a representation of some $GL_{n_i}$ (resp. $G_{n-n'}$). In particular, $\operatorname{Ind}_{P_{\textbf{s}}}^{G_n}$ is a functor from admissible representations of $M_{\textbf{s}}$ to admissible representations of $G_n$ that sends unitary representations to unitary representations.

In the case of $GL$, we denote the induced representation $\operatorname{Ind}_{P'}^{GL_n}(\rho_1 \otimes \cdots \otimes \rho_k)$ by
$$\rho_1 \times \cdots \times \rho_k$$
\noindent such that $P'=M'N'$ is a standard $F$-parabolic subgroup of $GL_n$ where $M' \cong GL_{n_1} \times GL_{n_2} \times $ $\cdots \times GL_{n_k}$ and each $\rho_i$ is a representation of $GL_{n_i}$ for $i=1, \ldots, k$. 

We also follow the notation in \cite{Bernstein_Zelevinsky:1977, Zelevinsky:1980}. Let $\rho$ be an irreducible unitary supercuspidal representation of some $GL_p$ and let $\nu$ be a character of $GL_p$ defined by $|\operatorname{det}|_F$. We define the segment, $\Delta:=[\nu^a \rho, \nu^{a+k} \rho] = \{\nu^a \rho, \nu^{a+1} \rho, \ldots \nu^{a+k} \rho\}$ where $a \in \mathbb{R}$ and $k \in \mathbb{Z}_{\geq 0}$. It is well known that the induced representation $\nu^{a+k} \rho \times \nu^{a+k-1} \rho \times \cdots \times \nu^{a} \rho$ has a unique irreducible subrepresentation, which we denote by $\delta(\Delta)$. The $\delta(\Delta)$ is an essentially square-integrable representation attached to $\Delta$ (The classification of discrete series representations in the case of general linear groups; see \cite{Zelevinsky:1980}, 3.1).

Let $\rho$ (resp. $\tau$) be an irreducible supercuspidal representation of $GL_m$ (resp. $G_n$). For a positive half-integer $\beta$ we say that ($\rho, \tau$) satisfies $(C(\beta))$ if $$\nu^{\beta} \rho \rtimes \tau \ \ \textit{is reducible and} \ \  \nu^{\gamma} \rho \rtimes \tau \ \ \textit{is irreducible for all} \ \ \gamma \in \mathbb{R} \ \  \textit{with} \ \ |\gamma| \neq \beta$$

We sometimes use the notations $\pi_{cusp}, \pi_{ds}$ and $\pi_t$ for supercuspidal, discrete series and tempered representations of $GSpin$ groups respectively. The local functorial lifts of $\pi_{cusp}, \pi_{ds}$ and $\pi_t$ are denoted by $\Pi_{cusp}, \Pi_{ds}$ and $\Pi_t$ respectively. 

For an irreducible representation $\pi$, the central character of $\pi$ is denoted by $\omega_{\pi}$.

\section{Preliminaries}
\label{Sec3}

\subsection{Classification of strongly positive representations of $GSpin$ groups }
\label{Sec3.1}

Let us briefly recall the results on the classification of strongly positive representations of $GSpin$ groups \cite{YKim:2015, YKim:2016}. Note that, in \cite{YKim:2015, YKim:2016}, the author closely follows the methods introduced in \cite{Matic:2011} and generalize those results to the case of $GSpin$ groups.

\begin{definition}[Strongly positive]
An irreducible admissible representation $\sigma$ of $G_n$ is called strongly positive if for each representation
$\nu^{s_1} \rho_1 \times \nu^{s_2} \rho_2 \times \cdots \times \nu^{s_k} \rho_k \rtimes \sigma_{cusp}$, where each $\rho_i, i=1, 2, \cdots, k$, is an irreducible supercuspidal unitary representation of some $GL_{n_i}$, $\sigma_{cusp}$ an irreducible supercuspidal representation of $G_{n'}$  and
$s_i \in \mathbb{R}$,  $ i=1, 2, \cdots,k$, such that
$$\sigma \hookrightarrow \nu^{s_1}\rho_1 \times \nu^{s_2}\rho_2 \times \cdots \times \nu^{s_k}\rho_k \rtimes \sigma_{cusp},$$
\noindent we have $s_i >0$ for each $i$.
\end{definition}

Let $SP$ denote the set of all strongly positive representations of $GSpin$ groups and let $LJ$ denote the set of $(Jord, \sigma')$, where $Jord= \displaystyle\bigcup\limits_{i=1}^{k}
\displaystyle\bigcup\limits_{j=1}^{k_i} \{ (\rho_i , b_j^{(i)}) \} $ and $\sigma'$ be an irreducible supercuspidal representation of $GSpin$ groups such that
\begin{enumerate}[(i)]
  \item $ \{ \rho_1, \rho_2, \ldots, \rho_k \}$ is a (possibly empty) set of mutually non-isomorphic irreducible and  essentially self-dual supercuspidal unitary representations of $GL$ such that $\nu^{a_{\rho_i}} \rho_i \rtimes \sigma'$ is reducible for $a_{\rho_i} > 0$ (defining $a_{\rho_i}$ due to the uniqueness of the reducibility point; See Remark for more details),
  \item $k_i = \lceil a_{\rho_i}\rceil$, and
  \item for each $i=1, 2, \ldots, k,$ a sequence of real numbers $b_1^{(i)}, b_2^{(i)}, \ldots, b_{k_i}^{(i)}$ satisfies $a_{\rho_i} - b_j^{(i)} \in \mathbb{Z}$, for $j=1, 2, \ldots, k_i$, and $-1 < b_1^{(i)} < b_2^{(i)} < \cdots < b_{k_i}^{(i)}$.
\end{enumerate}

The following is the main results in \cite{YKim:2015, YKim:2016}, i.e., classification of strongly positive representations of $GSpin$ groups:

\begin{thrm}\label{SPDS:SC}
 There exists a bijective mapping $\operatorname{\Phi}$ between $SP$ and $LJ$. More precisely, $\operatorname{\Phi}$ is constructed as follows: let $\sigma \in SP$. It can be considered to be the unique irreducible subrepresentation of the form 
\begin{align}\label{SPDS:induced}
( \displaystyle\prod\limits_{i=1}^{k} \displaystyle\prod\limits_{j=1}^{k_i} 
\de {a_{\rho_i}- k_i +j} {i} {b_j^{(i)}} {i} ) \rtimes \sigma'.
\end{align}
 Then, we define $\operatorname{\Phi}(\sigma)$ as $(\displaystyle\bigcup\limits_{i=1}^{k}
\displaystyle\bigcup\limits_{j=1}^{k_i} \{ (\rho_i , b_j^{(i)}) \}, \sigma') \in LJ. $
\end{thrm}

The strongly positive representations are special kinds of discrete series due to the Casselman’s square integrability criterion in \cite{WKim:2009}. Furthermore, the strongly positive representations can be considered to be basic building blocks for discrete series representations \cite{YKim:2015, YKim:2016}:

\begin{thrm}\label{DS:SPDS}
Let $\sigma$ denote a discrete series representation of $G_n$. Then there exists an embedding of the form
$$ \sigma \hookrightarrow \delta([\nu^{a_1} \rho_1, \nu^{b_1} \rho_1]) \times \delta([\nu^{a_2} \rho_2, \nu^{b_2} \rho_2]) \times \cdots \times
\delta([\nu^{a_r} \rho_r, \nu^{b_r} \rho_r]) \rtimes \sigma_{sp} $$
where $a_i \leq 0, a_i + b_i > 0$ and $ \rho_i$ is an irreducible unitary supercuspidal representation of $GL$ for $i=1, \ldots, r$, where $\sigma_{sp}$ is a strongly positive representation of $GSpin$ (we allow $k=0$).
\end{thrm}

\begin{remark}\label{remark:reducibility point}
The results on the classification of strongly positive representations for $GSpin$ groups \cite{YKim:2015, YKim:2016} depend on the assumption that there exists a unique non-negative real number $a$, such that $\nu^{a} \rho \rtimes \sigma$ reduces, where $\rho$ (resp. $\sigma$) denote an irreducible unitary supercuspidal representation of $GL_n$ (resp. $G_n$) \cite{Silberger:1980}. However, in the generic case, Shahidi proved that the reducibility points are unique and those are either 0, 1/2, or 1 \cite[Corollary 7.6 and Theorem 8.1]{Shahidi:1990a}. Therefore, the classification of strongly positive generic representations (Theorem \ref{classification:generic}) is unconditional.
\end{remark}

\subsection{$L$-functions from Langlands-Shahidi method in the case of $GSpin$ groups}
\label{Sec3.2}

Briefly, Langlands-Shahidi method defines several local $L$-functions that is attached to irreducible admissible generic representations of a connected reductive groups via the theory of intertwining operators \cite{Shahidi:1978, Shahidi:1981, Shahidi:1988, Shahidi:1990a, Shahidi:2010}. In the case of $GSpin$ groups, there are two types of $L$-functions. To describe those two $L$-functions more precisely, let $M \cong GL_m \times G_n$ be the Levi subgroup of a maximal standard $F$-parabolic subgroup $P=MN$ of $G_{m+n}$ and let $\pi' \otimes \pi$ be an irreducible admissible generic representation of $M$. The adjoint action $r$ of $\,^{L}{M}$, the $L$-group of $M$, on $\,^{L}\mathfrak{n}$, the Lie algebra of the $L$-group of $N$ decomposes as $r=r_2 \ \text{or} \ r_1 \oplus r_2$ \cite[Proposition 5.6]{Asgari:2002}, where
\noindent \[ \left\{ \begin{array}{cc}
 r_1= \rho_m \otimes \widetilde{R} \ \ \ \text{and} \ \ \ r_2=Sym^2 \otimes \mu^{-1} &  \text{if} \ \ \ G_n = GSpin_{2n+1} \\
  r_1= \rho_m \otimes \widetilde{R} \ \ \ \text{and} \ \ \ r_2=\wedge^2 \otimes \mu^{-1} &  \text{if} \ \ \ G_n = GSpin_{2n}
\end{array} \right. \]
\noindent Here $\rho_m$ is a standard representation of $GL_m(\mathbb{C})$, $\widetilde{R}$ is the contragredient of the standard representations of $\,^{L}{G}_n$ and $\mu$ is a similitude character of $\,^{L}{G}_n$. Using the existence of Whittaker model (generic representations), Shahidi defined so-called local coefficients and $\gamma$-functions. (See page 279 and Theorem 3.5 of \cite{Shahidi:1990a} for their definitions). In the case of $GSpin$ groups, the $\gamma$-functions that is attached to $\pi' \otimes\pi$ and $r_i$ for $i=1,2$ are denoted by $\gamma(s, \pi' \otimes \pi, r_i, \psi_{F})$, where $\psi_F$ is a fixed non trivial additive character of $F$ \cite[Theorem 3.5]{Shahidi:1990a}. Then, the $L$-functions from Langlands-Shahidi method that is attached to $\pi' \otimes \pi$ and $r_i$, denoted $L(s, \pi' \otimes \pi, r_i)$ for $i=1,2$, in \cite{Shahidi:1990a} are defined in the following way: When $\pi' \otimes \pi$ is tempered, Shahidi defined $L$-functions as an inverse of the normalized numerator of the $\gamma$-factors. Then, for an arbitrary irreducible admissible generic representation, he followed Langlands classification \cite{Silberger:1978} to define $L$-functions in general. The first $L$-function, i.e. $L(s, \pi' \otimes \pi, r_1)$, is the Rankin product $L$-function for $GL_m \times G_n$ which is denoted by $L(s, \pi' \times \widetilde{\pi})$ in \cite{Asgari_Shahidi:2006}, and the second $L$-function, i.e. $L(s, \pi' \otimes \pi, r_2)$, is either a twisted symmetric square $L$-function or a twisted exterior square $L$-function.

\begin{remark}\label{L:tempered}
In the tempered case, the equality of $\gamma$-factors implies the equality of $L$-functions since the $L$-functions are completely determined by $\gamma$-functions in the tempered case. For example, if we know the equality $\gamma(s, \pi_1' \otimes\pi_1, r, \psi_F) = \gamma (\pi_2' \otimes \pi_2, r, \psi_F)$ for tempered representations $\pi_1, \pi_2, \pi_1'$ and $\pi_2'$, we also have $L(s, \pi_1' \otimes \pi_1, r) = L(\pi_2' \otimes \pi_2, r)$.
\end{remark}

Let us recall structure theory for $GSpin$ groups which is studied by Asgari \cite{Asgari:2002}.

\begin{definition}\label{GSpin-def}
The split $GSpin$ groups $\textbf{G}_n:=\textbf{GSpin}_{2n+1}$ (resp. $\textbf{GSpin}_{2n}$) are split reductive algebraic groups of type $B_n$ (resp. $D_n$) whose derived subgroups are double coverings of split special orthogonal groups. Furthermore, the connected component of their Langlands dual groups are $\textbf{GSp}_{2n}(\mathbb{C})$ (resp. $\textbf{GSO}_{2n}(\mathbb{C})$).
\end{definition} 

\begin{prop}\label{GSpin-rootdatum}
The root datum $(X^*, R^*, X_*, R_*)$ of $\textbf{G}_n$ can be described as the following. 
$$X^*=\mathbb{Z}e_0 \oplus \mathbb{Z}e_1 \oplus \cdots \oplus \mathbb{Z}e_n,$$
$$X_*=\mathbb{Z}e_0^* \oplus \mathbb{Z}e_1^* \oplus \cdots \oplus \mathbb{Z}e_n^*.$$
(There is a standard $\mathbb{Z}$-pairing $< , >$ on $ X^* \times X_*$). 
And $R^*$ and $R_*$ are generated, respectively, by
$$\Delta^*= \{ \alpha_1 = e_1 - e_2, \alpha_2 = e_2 - e_3, \cdots, \alpha_{n-1}=e_{n-1}-e_n, \alpha_n=e_n \}, $$
$$\Delta_*= \{ \alpha_1^* = e_1^* - e_2^*, \alpha_2^* = e_2^* - e_3^*, \cdots, \alpha_{n-1}^*=e_{n-1}^*-e_n^*, \alpha_n^*=2e_n^*-e_0^* \}, $$
if $\textbf{G}_n:=\textbf{GSpin}_{2n+1}$ and by
$$\Delta^*= \{ \alpha_1 = e_1 - e_2, \alpha_2 = e_2 - e_3, \cdots, \alpha_{n-1}=e_{n-1}-e_n, \alpha_n=e_{n-1}+e_n \}, $$
$$\Delta_*= \{ \alpha_1^* = e_1^* - e_2^*, \alpha_2^* = e_2^* - e_3^*, \cdots, \alpha_{n-1}^*=e_{n-1}^*-e_n^*, \alpha_n^*=e_{n-1}^*+ e_n^*-e_0^* \}, $$
if $\textbf{G}_n:=\textbf{GSpin}_{2n}$.
\end{prop}

\begin{remark}\label{GSpin-dual}
The root datum of $\textbf{G}_n:=\textbf{GSpin}_{2n+1}$ (resp. $\textbf{GSpin}_{2n}$) is the dual root datum to the one for the group $\textbf{GSp}_{2n}$ (resp. $\textbf{GSO}_{2n}$). 
\end{remark}

\begin{prop}\label{GSpin-levi}
The standard Levi subgroups of $G_n$ are isomorphic to
$$GL_{n_1} \times GL_{n_2} \times \cdots \times GL_{n_k} \times G_{n-n'} $$
with $n' \leq n$ and $n-n' \neq 1$.
\end{prop}

\section{The equality of $L$-functions}
\label{Sec4}

To show the equality of $L$-functions (Theorem \ref{E:main:general}) we follow the filtration of admissible representations according to the growth properties of matrix coefficients:
\begin{align}\label{filtration}
\textit{supercuspidal } \subseteq \textit{ discrete series } \subseteq \textit{ tempered } \subseteq \textit{ admissible}.
\end{align}

\subsection{The supercuspidal case}
\label{Sec4.1}

In the supercuspidal case, we will briefly recall and sketch the proof in \cite{Heiermann_Kim:2017} to emphasize the difference with the case of discrete series representations. Let $\pi_{cusp}$ be an irreducible supercuspidal generic representation of $G_n$. The following is one main result in \cite{Heiermann_Kim:2017}:

\begin{thrm}\label{sc}
There exists a unique generic representation $\Pi_{cusp}$ of $GL_N$ such that for every irreducible supercuspidal representation $\rho$ of $GL_m$ we have
$$\gamma(s, \rho \times \pi_{cusp}, \psi_F) = \gamma(s, \rho \times \Pi_{cusp}, \psi_F) \ \ \text{and} \ \ L(s, \rho \times \pi_{cusp}) = L(s, \rho \times \Pi_{cusp}).$$
\end{thrm}

\begin{definition}\label{def:lift}
The representation $\Pi_{cusp}$ in Theorem \ref{sc} is called the local functorial lift of $\pi_{cusp}$.
\end{definition}

\begin{proof}
We will sketch the proof in \cite{Heiermann_Kim:2017} for completeness. Since $\pi_{cusp}$ is an irreducible supercuspidal generic representation of $G_n$, it can be embedded into a globally generic automorphic cuspidal representations $\underline{\pi}$ \cite[Proposition 5.1]{Shahidi:1990a}. Then, the global functional equation implies the following equality of $\gamma$-factors: 
$$\gamma(s, \rho \times \pi_{cusp}, \psi_F) = \gamma(s, \rho \times \Pi_{cusp}, \psi_F)$$
for every irreducible supercuspidal representation $\rho $ of $GL_m$.
The last step is to show that the functorial lift $\Pi_{cusp}$ of $\pi_{cusp}$ is also tempered (Remark \ref{F:main}). Then, the equality of $\gamma$-factors implies the equality of $L$-functions (Remark \ref{L:tempered}).
\end{proof}

\begin{remark}\label{F:main}
Let $\pi_{cusp}$ be an irreducible supercuspidal generic representation of the group $GSpin_{2n+1}$ $($resp. $GSpin_{2n})$. It is shown, in \cite{Heiermann_Kim:2017}, that its local functorial lift $\Pi_{cusp}$ is of the form
$$\Pi_{1} \times \cdots \times \Pi_{d}$$
where each $\Pi_{i}$ is an irreducible supercuspidal representation of some $GL_{2n_i}$ such that the twisted exterior square $L$-function $L(s, \Pi_{i} \otimes \omega_{\pi_{cusp}}, \wedge^2 \otimes \mu^{-1})$ $($resp. the twisted symmetric square $L$-function $L(s, \Pi_{i} \otimes \omega_{\pi_{cusp}} , Sym^2 \otimes \mu^{-1}))$ has a pole at $s=0$, $\widetilde{\Pi}_{i} \otimes (\omega_{\pi_{cusp}} \circ \operatorname{det}) \cong \Pi_{i}$ and $\Pi_{i} \ncong \Pi_{j}$ for $i \neq j$. In particular, $\Pi_{cusp}$ is a tempered representation. 
\end{remark}

\subsection{The general case}
\label{Sec4.2}

In a previous section, we state the results on the equality of $L$-functions in the case of supercuspidal representations (Theorem \ref{sc}). In this section, we generalize those results to the case of admissible representations. 

We first consider the case of discrete series generic representation, i.e., the second class in the filtration \eqref{filtration} of admissible representations. Let $\pi_{ds}$ be an irreducible discrete series generic representation of $G_n$. In this case, we use the following purely local result (a generic version of classification of (strongly positive) discrete series representations):

\begin{thrm}\label{classification:generic}
Let $\pi_{ds}$ be an irreducible discrete series generic representation of $G_n$. Then $\pi_{ds}$ can be embedded into the following induced representation:
\begin{align}\label{DS:SC:generic}
\pi_{ds} \hookrightarrow \delta_1 \times \cdots \times \delta_r \times \delta_1' \times \cdots \times \delta_k' \rtimes \pi_{cusp}.
\end{align}

where $\delta_i = \delta([\nu^{a_i} \rho_i, \nu^{b_i} \rho_i])$ are as in Theorem \ref{DS:SPDS} for $i=1, \ldots, r$, $\delta_l' =  \delta([\nu^{a_{\rho_l'}} \rho_l' ,\nu^{b^{(l)}} \rho_l']))$ are as in Theorem \ref{SPDS:SC} for $l=1, \ldots, k$, and $\pi_{cusp}$ is an irreducible supercuspidal generic representation of $G_{n'}$. Furthermore, each $a_{\rho_l'}$ is in $\{ \frac{1}{2}, 1 \}$ for $l = 1, \ldots, k$.
\end{thrm}

\begin{proof}
First, a combination of Theorem \ref{SPDS:SC} and Theorem \ref{DS:SPDS} implies the following embedding:
\begin{align}\label{embedding1}
\pi_{ds} \hookrightarrow \displaystyle\prod\limits_{i=1}^{r} \delta([\nu^{a_i} \rho_i, \nu^{b_i} \rho_i]) \times \displaystyle\prod\limits_{l=1}^{k} \displaystyle\prod\limits_{j=1}^{k_l} 
\delta([\nu^{a_{\rho_l'}-k_l +j} \rho_l', \nu^{b_j^{(l)}} \rho_l']) \rtimes \pi_{cusp}.
\end{align}

where $a_i, b_i$ and $\rho_i$ are as in Theorem \ref{DS:SPDS}, where $a_{\rho_l'}, b_j^{(l)}$ and $\rho_l'$ are as in Theorem \ref{SPDS:SC}, where $\pi_{cusp}$ is an irreducible supercuspidal generic representation of $G_n$. Let us recall that $a_{\rho_l'}$ is the reducibility point, i.e., $\nu^{s} \rho_l' \rtimes \pi_{cusp}$ is reducible if and only if $|s|=a_{\rho_l'}$. Due to the results of Shahidi \cite{Shahidi:1990a}, the reducibility point $a_{\rho_i'}$ in \eqref{embedding1} is in $\{ 0, \frac{1}{2}, 1 \}$, i.e., $(\rho_i', \pi_{cusp})$ satisfies $(C(0)), (C(\frac{1}{2})),$ or $(C(1))$. Since $a_{\rho_l'}$ is always positive, it is either $\frac{1}{2}$ or $1$. Therefore, $k_l = \lceil a_{\rho_l'}\rceil$, i.e., the smallest integer which is not smaller than $a_{\rho_l'}$, is $1$. In sum, we have the following embedding which completes the proof of the theorem:
$$\pi_{ds} \hookrightarrow \displaystyle\prod\limits_{i=1}^{r} \delta([\nu^{a_i} \rho_i, \nu^{b_i} \rho_i]) \times \displaystyle\prod\limits_{l=1}^{k}  
\delta([\nu^{a_{\rho_l'}} \rho_l', \nu^{b_1^{(l)}} \rho_l']) \rtimes \pi_{cusp}.$$
\end{proof}

Let $\Pi_{cusp}$ be the local functorial lift of $\pi_{cusp}$ as defined in Definition \ref{def:lift}. We consider the induced representation of $GL_N$ defined by
\begin{align}\label{eq:FL}
\Xi =\delta_1 \times \cdots \times \delta_r \times \delta_1' \times \cdots \times \delta_k' \times \Pi_{cusp} \times (\widetilde{\delta_k'} \otimes (\omega_{\pi_{ds}} \circ \operatorname{det}) ) \times \cdots \times (\widetilde{\delta_1} \otimes (\omega_{\pi_{ds}} \circ \operatorname{det})).
\end{align}
Then this induced representation has a unique generic constituent, denoted $\Pi_{ds}$ \cite{Zelevinsky:1980}.

Applying the multiplicativity of $\gamma$-factors in Langlands-Shahidi method \cite{Shahidi:1990b} to the embedding \eqref{DS:SC:generic} in Theorem \ref{classification:generic}, we have
$$\gamma(s, \rho \times \pi_{ds}, \psi_F) =  \gamma(s, \rho \times \pi_{cusp}, \psi_F) \ \times $$
$$\displaystyle\prod\limits_{i=1}^{r} \gamma(s , \rho \times \delta_i, \psi_F) \gamma(s, \rho \times (\widetilde{\delta_i}\otimes \omega_{\pi_{ds}}), \psi_F) \displaystyle\prod\limits_{l=1}^{k} \gamma(s , \rho \times \delta_l', \psi_F) \gamma(s, \rho \times (\widetilde{\delta_l'}\otimes \omega_{\pi_{ds}}), \psi_F).$$

In the general linear groups side, i.e., $\Pi_{ds}$, we can also apply the multiplicativity of $\gamma$-factors \cite{Jacquet_PS_Shalika:1983} to get the same expression of $\gamma(s, \rho \times \Pi_{ds}, \psi_F)$ as above. Then, Theorem \ref{sc} implies

$$\gamma(s, \rho \times \pi_{ds}, \psi_F) = \gamma(s, \rho \times \Pi_{ds}, \psi_F).$$

To obtain the equality of $L$-functions, our next step is to show that $\Pi_{ds}$ is a tempered representation as in the case of supercuspidal representations. However, as we explained in the introduction, the approach is of totally different nature. Here, we use a purely local result (Theorem \ref{classification:generic}) to show that $\Pi_{ds}$ is tempered. 

\begin{remark}
We apply the idea used in \cite{Cogdell_Kim_PS_Shahidi:2004} to the case of $GSpin$ groups. Let us briefly explain main differences between the classical group case and $GSpin$ case. First, the structure is different and more complicated. For example, we need to deal with essentially self-dual representations due to Weyl group action on the Levi subgroups (See Corollary 3.2 in \cite{YKim:2015} for more details). In other words, we need to control the central character. Second, the classification of discrete series of $GSpin$ groups is not fully known, while classical groups is fully known \cite{Moeglin_Tadic:2002}. (Note that Mati\'c and the author recently obtain a full classification of discrete series representations of odd $GSpin$ groups \cite{Kim_Matic:2017}). 
\end{remark}
\begin{prop}\label{ds} 
Let $\pi_{ds}$ be a discrete series generic representation of $G_n$ as in Theorem \ref{classification:generic}. Then $\pi_{ds}$ has a local functorial lift $\Pi_{ds}$ to $GL_N$ such that for every supercuspidal representation $\rho$ of $GL_m$ we have
$$L(s, \rho \times \pi_{ds})=L(s, \rho \times \Pi_{ds}) \ \ \ \ \text{ and } \ \ \ \  \gamma(s, \rho \times \pi_{ds}, \psi_F) =  \gamma(s, \rho \times \Pi_{ds}, \psi_F).$$
\end{prop}

\begin{proof} 
We explicitly compute $\Pi_{ds}$ using induction in stages as in \cite{Cogdell_Kim_PS_Shahidi:2004}. Let us first consider $\delta_i \otimes (\widetilde{\delta_i} \otimes (\omega_{\pi_{ds}} \circ \operatorname{det}))$ for $ 1 \leq i \leq r$, which is a supercuspidal support of $\Pi_{ds}$. Here, $\delta_i=\delta([\nu^{a_i}\rho_i, \nu^{b_i}\rho_i])$ is the unique irreducible subrepresentation of $\nu^{b_i}\rho_i \times \nu^{b_i-1}\rho_i \times \cdots \times \nu^{a_i}\rho_i$ such that $a_i \leq 0$ and $a_i+b_i \geq 0$ (Theorem \ref{classification:generic}). Then, replacing $\delta_i$ by its inducing data $\nu^{b_i}\rho_i \times \nu^{b_i-1}\rho_i \times \cdots \times \nu^{a_i}\rho_i$, $\delta_i \times (\widetilde{\delta_i} \otimes (\omega_{\pi_{ds}} \circ \operatorname{det}))$ becomes

$$\nu^{b_i}\rho_i \times \nu^{b_i-1}\rho_i \times \cdots \times \nu^{a_i}\rho_i \times (\nu^{-a_i}\widetilde{\rho_i} \otimes (\omega_{\pi_{ds}} \circ \operatorname{det}) )   \times \cdots \times (\nu^{-b_i}\widetilde{\rho_i} \otimes (\omega_{\pi_{ds}} \circ \operatorname{det})) $$
$$\cong \nu^{b_i}\rho_i \times \nu^{b_i-1}\rho_i \times \cdots \times \nu^{a_i}\rho_i \times \nu^{-a_i}\rho_i \times \nu^{-a_i-1}\rho_i \times \cdots \times \nu^{-b_i}\rho_i. $$
Furthermore, by rearranging the inducing data, we have
\begin{align}\label{constituent1}
\nu^{b_i} \rho_i \times \nu^{b_i-1} \rho_i \times \cdots \times \nu^{-b_i} \rho_i  \times \nu^{-a_i}\rho_i \times \nu^{-a_i-1}\rho_i \times \cdots \times \nu^{a_i}\rho_i.
\end{align}

The generic constituent of the induced representation \eqref{constituent1} is a tempered representation $\delta([\nu^{-b_i}\rho_i, \nu^{b_i}\rho_i]) \times \delta([\nu^{a_i}\rho_i, \nu^{-a_i}\rho_i])$. Therefore, when we compute the generic constituent of $\Xi$ in \eqref{eq:FL}, we replace each $\delta_i \otimes (\widetilde{\delta_i} \otimes (\omega_{\pi_{ds}} \circ \operatorname{det}))$ by $\delta([\nu^{-b_i}\rho_i, \nu^{b_i}\rho_i]) \otimes \delta([\nu^{a_i}\rho_i, \nu^{-a_i}\rho_i])$ for $1 \leq i \leq r$.

Next, we consider a supercuspidal support $\delta_l' \otimes (\widetilde{\delta_l'} \otimes (\omega_{\pi_{ds}} \circ \operatorname{det}))$ for $ 1 \leq l \leq k$. In this case, we need to split it into two cases depending on the associated exponent $a_{\rho_l'}$. First, let us consider the case when $a_{\rho_l'}$ is $\frac{1}{2}$, i.e., when $(\rho_l', \pi_{cusp})$ satisfies $C(\frac{1}{2})$. Here, $\delta_l' = \delta([\nu^{\frac{1}{2}}\rho_l', \nu^{b^{(l)}}\rho_l'])$ is the unique irreducible subrepresentation of $\nu^{b^{(l)}}\rho_l' \times \nu^{b^{(l)}-1}\rho_l' \times \cdots \times \nu^{\frac{1}{2}}\rho_l'$ (Theorem \ref{classification:generic}). As in the previous case, replacing $\delta_l'$ by its inducing data, $\delta_l' \times (\widetilde{\delta_l'} \otimes (\omega_{\pi_{ds}} \circ \operatorname{det}))$ becomes 
$$\nu^{b^{(l)}} \rho_l' \times \cdots \times \nu^{\frac{3}{2}}\rho_l' \times \nu^{\frac{1}{2}}\rho_l' \times \nu^{-\frac{1}{2}}\rho_l' \times \nu^{-\frac{3}{2}}\rho_l' \times \cdots \times \rho_l' \nu^{-b^{(l)}}. $$

Therefore, we replace $\delta_l' \otimes (\widetilde{\delta_l'} \otimes (\omega_{\pi_{ds}} \circ \operatorname{det}) )$ by $\delta([\nu^{-b^{(l)}}\rho_l', \nu^{b^{(l)}}\rho_l'])$ in the inducing data when we compute the generic constituent of $\Xi$.

Finally, let us consider the case when $a_{\rho_l'}$ is $1$, i.e., when $(\rho_l', \pi_{cusp})$ satisfies $(C(1))$. In this case, $\delta_l' = \delta([\nu^1\rho_l', \nu^{b^{(l)}}\rho_l'])$. As in the previous case, $\delta_l' \times (\widetilde{\delta_l'} \otimes (\omega_{\pi_{ds}} \circ \operatorname{det}))$ is a constituent of the larger induced representation
\begin{align}\label{C(1)case}
\nu^{b^{(l)}}\rho_l' \times \cdots \times \nu^{2}\rho_l' \times \nu^{1}\rho_l' \times \nu^{-1}\rho_l' \times \nu^{-2}\rho_l' \times \cdots \times \nu^{-b^{(l)}}\rho_l'.
\end{align}
(Note that $\nu^{0}\rho_l'$ is missing between $\nu^{1}\rho_l'$ and $\nu^{-1}\rho_l'$). 

Not as in the previous case, we cannot replace $\delta_l' \times (\widetilde{\delta_l'} \otimes (\omega_{\pi_{ds}} \circ \operatorname{det}))$ by discrete series representation since $\nu^{0}\rho_l'$ is missing in the inducing data of the induced representation \eqref{C(1)case}. To overcome this situation, we borrow such $\rho_l'$ from $\Pi_{cusp}$ as follows:
By Remark \ref{F:main}, we know that $\Pi_{cusp}$ (the local functorial lift of $\pi_{cusp}$) is of the form
$$\Pi_{1} \times \cdots \times \Pi_{d}$$
with each $\Pi_{i}$ a supercuspidal representation of an appropriate general linear group $GL_{d_i}$ and the $\Pi_{i}$ distinct. Therefore, the set $\{ \Pi_1, \cdots, \Pi_d \}$ are precisely the set of supercuspidal representation $\rho$ for which $L(s, \widetilde{\rho} \times \Pi_{cusp})$ has a pole at $s=0$. On the other hand, Theorem 8.1 of \cite{Shahidi:1990a} implies that if $(\rho_l', \pi_{cusp})$ satisfies $(C(1))$, then $L(s, \widetilde{\rho_l'} \times \pi_{cusp}) = L(s, \widetilde{\rho_l'} \times \Pi_{cusp})$ has a pole at $s=0$. Hence, if we let $J$ be the set of $m \in \{1, 2, \ldots, k \}$ such that $a_{\rho_m'}=1$, then $\rho_m' \in \{ \Pi_{1}, \ldots, \Pi_{d} \}$ for $m \in J$. Let us write 
$$\Pi_{cusp}= \displaystyle\prod\limits_{m \in J} \rho_m' \times \displaystyle\prod\limits_{n \notin J} \Pi_n.$$ 
Now, for each $m \in J$, let us consider $\delta_m' \otimes \rho_m' \otimes (\widetilde{\delta_m'} \otimes (\omega_{\pi_{ds}} \circ \operatorname{det}))$. If we replace $\delta_m'$ by its inducing data, we see that $\delta_m' \times \rho_m' \times (\widetilde{\delta_m'} \otimes (\omega_{\pi_{ds}} \circ \operatorname{det}) ))$ is a constituent of the larger induced representation
$$\nu^{b^{(m)}}\rho_m' \times \cdots \times \nu\rho_m' \times \rho_m' \times \nu^{-1}\rho_m' \times \cdots \times \nu^{-b^{(m)}}\rho_m'$$
and this representation has as its unique generic constituent the discrete series representation $\delta([\nu^{-b^{(m)}}\rho_m', \nu^{b^{(m)}}\rho_m'])$. Therefore, in the inducing data $\Xi$, we replace $\delta_m' \otimes \rho_m' \otimes (\widetilde{\delta_m'} \otimes (\omega_{\pi_{ds}} \circ \operatorname{det}) )$ by $\delta([\nu^{-b^{(m)}}\rho_m', \nu^{b^{(m)}}\rho_m'])$.

In sum, we conclude that the lift $\Pi_{ds}$ becomes the unique generic constituent of the induced representation 
$$\delta([\nu^{-b_1}\rho_1, \nu^{b_1}\rho_1]) \times \delta([\nu^{a_1}\rho_1, \nu^{-a_1}\rho_1]) \times \cdots \times \delta([\nu^{-b_r}\rho_r, \nu^{b_r}\rho_r]) \times \delta([\nu^{a_r}\rho_r, \nu^{-a_r}\rho_r])$$ 
$$\times \delta([\nu^{-b^{(1)}}\rho_1', \nu^{b^{(1)}}\rho_1']) \times \cdots \times \delta([\nu^{-b^{(k)}}\rho_k', \nu^{b^{(k)}}\rho_k']) \times \displaystyle\prod\limits_{n \notin J} \Pi_n.$$ 

This full induced representation is irreducible since its inducing data\\ $\delta([\nu^{-b_1}\rho_1, \nu^{b_1}\rho_1]) \otimes \delta([\nu^{a_1}\rho_1, \nu^{-a_1}\rho_1]) \otimes \cdots \otimes \delta([\nu^{-b_r}\rho_r, \nu^{b_r}\rho_r]) \otimes \delta([\nu^{a_r}\rho_r, \nu^{-a_r}\rho_r])$ $\otimes \delta([\nu^{-b^{(1)}}\rho_1', \nu^{b^{(1)}}\rho_1']) \otimes$ $\cdots \otimes \delta([\nu^{-b^{(k)}}\rho_1', \nu^{b^{(k)}}\rho_1']) \otimes \bigotimes\limits_{n \notin J} \Pi_n$ is unitary and irreducible. 

Therefore, $\Pi_{ds}$ is exactly this full induced representation. Furthermore, $\Pi_{ds}$ is tempered since its inducing data is a unitary discrete series representation. Therefore, we also have the following equality of $L$-functions (Remark \ref{L:tempered}): 
$$L(s, \rho \times \pi_{ds})=L(s, \rho \times \Pi_{ds})$$
for any supercuspidal representation $\rho$ of $GL$. 
\end{proof}

Now, let us consider the case when $\pi_t$ is a tempered generic representation of the group $G_n$, i.e., the third class in the filtration \eqref{filtration}. 

\begin{prop}\label{tempered}
Let $\pi_t$ be a tempered generic representation of $G_n$. There exists a local functorial lift $\Pi_t$ to $GL_N$ such that for every supercuspidal representation $\rho$ of $GL_m$ we have
$$L(s, \rho \times \pi_t)=L(s, \rho \times \Pi_t) \text{ and } \gamma(s, \rho \times \pi_t, \psi_F) =  \gamma(s, \rho \times \Pi_t, \psi_F).$$

Furthermore, the lift $\Pi_t$ is irreducible, tempered and generic. 
\end{prop}

\begin{proof}

We can view $\pi_t$ as a direct summand of an induced representation $\delta_{1} \times \cdots \times \delta_{m} \times \pi_{ds}$ with each $\delta_{i}$ is discrete series representation of some $GL$ for $i=1, \ldots , m$ and $\pi_{ds}$ is a discrete series generic representation of $G_{n_0}$ for some $n_0 < n$.

We consider the following lift:
$$\Pi_t:=\delta_{1} \times \cdots \times \delta_{m} \times \Pi_{ds} \times (\widetilde{\delta_{m}} \otimes (\omega_{\pi_t} \circ \operatorname{det}) ) \times \cdots \times (\widetilde{\delta_{1}} \otimes (\omega_{\pi_t} \circ \operatorname{det}) )$$
where $\Pi_{ds}$ is the local functorial lift of $\pi_{ds}$ as in Proposition \ref{ds}. This representation is irreducible since its inducing data is unitary and irreducible. Furthermore, $\Pi_t$ is tempered and generic since its inducing data is a discrete series representation. Using the multiplicativity of $\gamma$-factors and the definition of $L$-functions in the tempered case, we can show as in the proof of Proposition \ref{ds} the following equalities of local factors:
$$L(s, \rho \times \pi_t)=L(s, \rho \times \Pi_t) \text{ and } \gamma(s, \rho \times \pi_t, \psi_F) =  \gamma(s, \rho \times \Pi_t, \psi_F)$$
for any supercuspidal representation $\rho$ of $GL$.
\end{proof}

\begin{cor}
In Proposition \ref{tempered}, the local functorial lift of $\pi_t$ is unique.
\end{cor}
\begin{proof}
Let $\Pi_1$ and $\Pi_2$ be two local functorial lifts of $\pi_t$. We have
$$\gamma(s, \rho \times \Pi_1, \psi_F) =  \gamma(s, \rho \times \pi_t, \psi_F)=\gamma(s, \rho \times \Pi_2, \psi_F)$$
for any irreducible supercuspidal representation $\rho$ of $GL$.
Then, the local converse theorem for $GL$ (Theorem 1.1 in \cite{Henniart:1993}) implies  
$\Pi_1 \cong \Pi_2.$
\end{proof}
Finally, let us consider the case when $\pi$ is an irreducible admissible generic representation of $G_n$.

\begin{thrm}\label{E:main}
Let $\pi$ be an irreducible admissible generic representation of $G_n$. 
There exists a local functorial lift $\Pi$ to $GL_N$ such that for every supercuspidal representation $\rho$ of $GL_m$ we have
$$L(s, \rho \times \pi)=L(s,\rho \times \Pi) \text{ and } \gamma(s, \rho \times \pi, \psi_F) =  \gamma(s, \rho \times \Pi, \psi_F).$$
\end{thrm}

\begin{proof}
Due to the standard module conjecture for $GSpin$ groups \cite{Heiermann_Muic:2007, WKim:2009}, we can view $\pi$ as a full induced representation of the following form:
$$\tau_{1} \nu^{r_1} \times \cdots \times \tau_{m}\nu^{r_m} \rtimes \pi_{t}$$
where each $\tau_{i}$ is a tempered representation of an appropriate $GL_{n_i}$, $0 < r_m < \cdots < r_1$, and $\pi_{t}$ is a tempered generic representation of some ${G}_{n_0}$ of the same type.

We consider the following lift:
$$\tau_{1} \nu^{r_1} \times \cdots \times \tau_{m}\nu^{r_m} \times \Pi_{t} \times (\widetilde{\tau_{m}} \otimes (\omega_{\pi_t} \circ \operatorname{det}) \nu^{-r_m}) \times \cdots \times (\widetilde{\tau_{1}} \otimes (\omega_{\pi_t} \circ \operatorname{det})\nu^{-r_1})$$
where $\Pi_{t}$ is the local functorial lift of $\pi_{t}$ as in Proposition \ref{tempered}. 

This induced representation has a unique irreducible quotient, i.e., Langlands quotient which we denote by $\Pi$.

We show the equality of $L$- and $\gamma$-factors.

On the general linear group side, by the multiplicativity of $\gamma$- and $L$-factors,
$$\gamma(s, \rho \times \Pi, \psi_F)=\gamma(s, \rho \times \Pi_{t}, \psi_F) \ \times $$
$$\displaystyle\prod\limits_{j=1}^m \gamma(s + r_j, \rho \times \tau_{j}, \psi_F)\gamma(s - r_j, \rho \times (\widetilde{\tau_{j}} \otimes (\omega_{\pi_t} \circ \operatorname{det})) \times \rho, \psi_F)$$

and
$$L(s, \rho \times \Pi)=L(s, \rho \times \Pi_{t}) \displaystyle\prod\limits_{j=1}^m L(s + r_j, \rho \times \tau_{j})L(s - r_j, \rho \times (\widetilde{\tau_{j}} \otimes (\omega_{\pi_t} \circ \operatorname{det}))).$$

On the $GSpin$ group side, by the multiplicativity of $\gamma$-factors,

$$\gamma(s, \rho \times \pi, \psi_F)=\gamma(s, \rho \times \pi_{t}, \psi_F) $$
$$\displaystyle\prod\limits_{j=1}^m \gamma(s + r_j, \rho \times \tau_{j}, \psi_F)\gamma(s - r_j, \rho \times (\widetilde{\tau_{j}} \otimes (\omega_{\pi_t} \circ \operatorname{det})), \psi_F).$$

By the Langlands classification, $L(s, \rho \times \pi)$ ($L$-function from Langlands-Shahidi method) is defined by 
$$L(s, \rho \times \pi_{t}) \displaystyle\prod\limits_{j=1}^m L(s + r_j, \rho \times \tau_{j})L(s - r_j, \rho \times (\widetilde{\tau_{j}} \otimes (\omega_{\pi_t} \circ \operatorname{det}))).$$

Therefore, this reduces us to show the following equalities (the case of tempered representations):
$$L(s, \rho \times \pi_{t})= L(s, \rho \times \Pi_{t}) \text{ and } \gamma(s, \rho \times \pi_{t}, \psi_F)= \gamma(s, \rho \times \Pi_{t}, \psi_F)$$

Then, theorem follows since this is exactly the case in Proposition \ref{tempered}.
\end{proof}

\begin{thrm}\label{E:main:general}
Let $\pi$ and $\Pi$ be as in Theorem \ref{E:main}. Then, for every irreducible admissible generic representation $\pi'$ of $GL_m$ we have
$$L(s,  \pi' \times \pi)=L(s, \pi' \times \Pi) \text{ and } \gamma(s, \pi' \times \pi, \psi_F) =  \gamma(s, \pi' \times \Pi, \psi_F).$$
\end{thrm}

\begin{proof}
Using the same argument as in the proof of Theorem \ref{E:main}, i.e., the standard module conjecture and the Langlands classification it is enough to show the theorem when $\pi$, $\Pi$ and $\pi'$ are tempered representations. We write $\pi'$ as 
$$\pi'=\delta_{1}' \times \cdots \times \delta_{k}'.$$
with each $\delta_{i}'$ discrete series for $i=1, \ldots, k$. Furthermore, each $\delta_{i}'$ can be considered as the unique generic subrepresentation of the induced representation
$$\nu^{t} \rho' \times \nu^{t-1}\rho' \times \cdots \times \nu^{-t}\rho'.$$
with $\rho'$ supercuspidal representation and $t$ a non-negative half integer.

Using the multiplicativity of $\gamma$-factors and Proposition \ref{tempered}, we have the equality of $\gamma$-factors. Furthermore, the equality of $L$-functions automatically follows since the representations $\pi$, $\Pi$ and $\pi'$ are tempered.
\end{proof}

\begin{remark}
In \cite{Heiermann_Kim:2017}, Heiermann and the author proved Theorem \ref{E:main:general} with different approach following Heiemrann's construction of Langlands parameters \cite{Heiermann:2004, Heiermann:2006, Heiermann:2006b}. In the present paper, we give another proof of Theorem \ref{E:main:general} using classification results Section \ref{Sec3.1} or \cite{YKim:2015, YKim:2016}. However, the classification results behaves an important role when we strengthen the the generic Arthur packet conjecture (see Remark \ref{Heiermann and Kim} for more details).
\end{remark}

\section{The generic $L$-packets and the generic Arthur packet conjecture}
\label{Sec5}

We first state an application of our main results in Section \ref{Sec4}, weak version of the generic Arthur packet conjecture for $GSpin$ groups. And to remove the assumption of the conjecture further, we define and study $L$-packets in Section \ref{Sec5.2}. Furthermore, using the properties of our $L$-packets that are obtained in \ref{Sec5.2}, we strengthen and prove the conjecture, called strong version of the generic Arthur packet conjecture for $GSpin$ groups.

\subsection{A weak version of the generic Arthur packet conjecture}
\label{Sec5.1}
For a connected reductive group ${\bf G}$ defined over non-archimedean local field $F$ of characteristic zero, let $L_{F}$ be the Weil-Deligne group of $F$, denoted $W_{F}'$, and let $\Phi(G)$ be the set of Langlands parameters, i.e., equivalence classes of homomorphisms
$$\phi: L_{F} \rightarrow \,^{L}G=\hat{G} \rtimes L_{F}$$
under conjugation by elements in $\hat{G}$, satisfying some several conditions (see \cite{Shahidi:2011}). We also define $\Phi_{temp}(G) \subset \Phi(G)$ as the set of $\phi$ such that the image $\phi(W_F)$ of $W_F$ in $\hat{G}$ is bounded. An element in $\Phi_{temp}(G)$ is called a tempered Langlands parameter. The Arthur parameter is defined by the following way:
Let $\Psi(G)$ be the set of $\hat{G}$-orbits of maps
$$\psi: L_{F} \times SL_2(\mathbb{C}) \rightarrow \,^{L}G=\hat{G} \rtimes L_{F}$$
such that $\psi|L_{F} \in \Phi_{temp}(G)$. 

For each $\psi \in \Psi(G)$, we attach the Langlands parameter $\phi_{\psi} \in \Phi(G)$ by

\begin{equation}\label{Langlands parameter that corr. to Arthur}
\phi_{\psi}(w)=\psi(w, \begin{pmatrix} |w|^{\frac{1}{2}} & 0 \\ 0 & |w|^{-\frac{1}{2}} \end{pmatrix}).
\end{equation}

In \cite{Shahidi:2011}, Shahidi proved the following theorem:

\begin{thrm}\label{ConjG}
For a connected reductive group ${\bf G}$ defined over $F$, we assume that there exists an $L$-packet $\Pi(\phi_{\psi})$ that corresponds to the Langlands parameter $\phi_{\psi}$, where $\phi_{\psi}$ corresponds to the Arthur parameter $\psi \in \Psi(G)$ as in (\ref{Langlands parameter that corr. to Arthur}). 
We further assume that the $L$-functions from Langlands-Shahidi method for a connected reductive group ${\bf G}$ are equal to Artin $L$-functions. If $\Pi(\phi_{\psi})$ has a generic member, then the Langlands parameter $\phi_{\psi}$ is tempered.\\
\end{thrm}

Now, we introduce a weak version of the generic Arthur packet conjecture.

\begin{conjecture}\label{Weak G}
Let $\psi, \phi_{\psi},$ and $\Pi(\phi_{\psi})$ be as in Theorem \ref{ConjG}. Assume that $\Pi(\phi_{\psi})$ has a generic member. Then, the Langlands parameter $\phi_{\psi}$ is tempered.
\end{conjecture}
Our main Theorem, i.e., Theorem \ref{E:main:general}, implies that we can remove the assumption on the equality of $L$-functions of Theorem \ref{ConjG} in the case of $GSpin$ groups. (Note that Henniart \cite{Henniart:2010} prove that the second $L$-functions in the $GSpin$ case, i.e., twisted symmetric square or twisted exterior square $L$-functions are equal to Artin $L$-functions).  In other words, we prove a weak version of the generic Arthur packet conjecture (Conjecture \ref{Conjecture A} in Introduction) in the case of $GSpin$ groups.
\begin{thrm}\label{ConjG:GSpin}
Assume that there exists an $L$-packet $\Pi(\phi_{\psi})$ that corresponds to the Langlands parameter $\phi_{\psi}$, where $\phi_{\psi}$ corresponds to the Arthur parameter $\psi$ for $GSpin$ groups. If $\Pi(\phi_{\psi})$ has a generic member, then $\phi_{\psi}$ is a tempered Langlands parameter.
\end{thrm}

\begin{remark}
In the case of classical groups, the Conjecture \ref{Weak G} is also studied and proved by Ban, Jantzen-Liu, and Liu \cite{Ban:2006, Jantzen_Liu:2014, Liu:2011}.
\end{remark}

\subsection{Construction of generic $L$-packets}
\label{Sec5.2}

Heiermann and the author constructed Langlands parameters that correspond to irreducible admissible generic representations of $GSpin$ (the generic local Langlands correspondence). More precisely, the following is one main result in \cite{Heiermann_Kim:2017}: 

\begin{thrm}\label{HeK}
Let $\pi$ be an irreducible admissible generic represetnation of $GSpin$. Then there exists a Langlands parameter $\phi_{\pi}$ such that for any irreducible admissible generic representation $\rho$ of $GL$, we have
$$L(s, \rho \times \pi) = L(s, \phi_{\rho} \otimes \phi_{\pi}) \textit{\ \ and \ \ } \gamma(s, \rho \times \pi, \psi_F) = \gamma(s, \phi_{\rho} \otimes \phi_{\pi}, \psi_F)$$
where $\phi_{\rho}$ is the Langlands parameter that corresponds to $\rho$ through the local Langlands correspondence for $GL$ \cite{Harris_Taylor:2001, Henniart:2000}.
\end{thrm}

The next natural question in the Langlands program is whether we can describe the set of irreducible admissible representations that correspond to the Langlands parameter $\phi_{\pi}$, i.e., the $L$-packet, where $\phi_{\pi}$ is the Langlands parameter that is constructed in Theorem \ref{HeK}. To answer the question, we first describe all generic representations in $L$-packets in terms of $L$-functions from Langlands-Shahidi method. Using the results in \cite{Heiermann_Kim:2017}, we derive the following theorem:

\begin{thrm}\label{generic reps in $L$-packet}
Let $\pi_1$ (resp. $\pi_2$) be an irreducible admissible generic representation of $GSpin$ and let $\phi_{\pi_1}$ (resp. $\phi_{\pi_2}$) be its Langlands parameter that is constructed in \cite{Heiermann_Kim:2017}. Then, $\phi_{\pi_1} \cong \phi_{\pi_2}$ if and only if for each irreducible admissible generic representation $\rho$ of $GL$ we have
$$L(s, \rho \times \pi_1) = L(s, \rho \times \pi_2) \textit{\ and \ } \gamma(s, \rho \times \pi_1, \psi_F) = \gamma(s, \rho \times \pi_2, \psi_F). $$ 
\end{thrm}

\begin{proof}
Theorem 4.9 of \cite{Heiermann_Kim:2017} implies the necessary condition. 
Now let us assume that $\pi_1$ and $\pi_2$ share the same local factors. Let $\Pi_i$ be the local functorial lift of $\pi_i$ for $i=1,2$. Then, for any irreducible admissible generic representation $\rho$ of $GL$ we have
$$\gamma(s, \rho \times \Pi_1, \psi_F) = \gamma(s, \rho \times \pi_1, \psi_F)=\gamma(s, \rho \times \pi_2, \psi_F) = \gamma(s, \rho \times \Pi_2, \psi_F).$$ 
Then, the local converse theorem for $GL$ \cite{Henniart:1993} implies that $\Pi_1 \cong \Pi_2$. In \cite[Theorems 4.8 and 4.9]{Heiermann_Kim:2017}, recall that the Langlands parameter of $\pi_i$ is constructed by the Langlands parameter of its local functorial lifts $\Pi_i$ for $i=1,2$. Therefore, we have $\phi_{\pi_1} \cong \phi_{\pi_2}.$
\end{proof}

Next, we describe all non-generic representations in $L$-packets. Following Theorems \ref{HeK} and \ref{generic reps in $L$-packet}, we define $L$-packets as follows:

\begin{definition}[Generic $L$-packets]\label{Def:generic L-packet}
For any fixed irreducible admissible generic representation $\pi$ of $GSpin$, an $L$-packet that contains $\pi$ is defined as the set of irreducible admissible representations $\pi'$ such that for each irreducible admissible representation $\rho$ of $GL$ we have
$$L(s, \rho \times \pi') = L(s, \rho \times \pi) \textit{\ and \ } \gamma(s, \rho \times \pi', \psi_F) = \gamma(s, \rho \times \pi, \psi_F). $$ 
We call such $L$-packets as generic $L$-packets.
\end{definition}

\begin{remark}\label{non-generic}
Definition of a generic $L$-packet is not well defined when it contains a non-generic representation since local factors are not defined in general. 
\end{remark}

To overcome the situation in Remark \ref{non-generic}, we need to assume the existence of Rankin product $L$-functions that correspond to non-generic representations. To establish the strong version of the generic Arthur packet conjecture, we assume that those new $L$-functions (Rankin product $L$-functions) that correspond to any admissible representations satisfy the following property:

\begin{assumption}\label{Assumption:LS}
The new $L$-functions that correspond to any admissible representations follow the Langlands classification.  
\end{assumption}

\begin{remark}\label{Harris}

\begin{enumerate}[(i)]
\item The property of $L$-functions in Assumption \ref{Assumption:LS} is very natural since the Langlands classification is well-constructed for any connected reductive groups \cite{Silberger:1978}. (It is also true that all $L$-functions that have been defined (including $L$-functions from Langlands-Shahidi) follow the Langlands classification).
\item $L$-packets for $GSpin$ groups is described by M$\oe$glin in \cite{Moeglin:2014} using the trace formula. Therefore, our result is unconditional if we use M$\oe$glin's result. However, our goal is to define and study $L$-packets using Langlands-Shahidi method, i.e., without using trace formula. Therefore, we keep writing the property (Assumption \ref{Assumption:LS}) as assumption.
\item One important property of $L$-packets that we need to show is that $L$-packets are not empty sets. More precisely, starting from Langlands parameters that satisfy several conditions, we need to show that there exists an irreducible generic admissible representation of $GSpin$ groups that corresponds to the Langlands parameters. This problem is directly related to the local descent method for $GSpin$ groups. This question was raised and commented by M. Harris and M. Asgari while the author was giving a talk at the 2016 Paul J. Sally, Jr. Midwest Representation Theory Conference in honor of the 70th Birthday of Philip Kutzko. Currently Lau, and Kaplan-Lau-Liu are working on the local descent method for odd and even $GSpin$ groups, respectively.
\end{enumerate}
\end{remark}

\begin{definition}
An $L$-packet $\Pi$ is called a tempered $L$-packet if all its members are tempered representations.
\end{definition}

We prove the following property for generic $L$-packets in the case of $GSpin$ groups (note that the following lemma is one main tool when we prove the strong version of the generic Arthur packet conjecture, Theorem \ref{StrongConjG}): 

\begin{lem}\label{tempered $L$-packet}
If an L-packet $\Pi_{\phi}$ for $GSpin$ contains a tempered representation which is also generic, all other members in $\Pi_{\phi}$ are tempered as well, i.e., $\Pi_{\phi}$ is a tempered $L$-packet.
\end{lem}

\begin{proof}
Let $\pi$ be an irreducible tempered generic representation of $G_n$ in $\Pi_{\phi}$. Suppose that there exists a non-tempered representation $\pi' \in \Pi_{\phi}$. Since $\pi'$ and $\pi$ are in the same $L$-packet, we have
\begin{align}\label{non-tempered in L-packet}
L(s, \rho \times \pi') = L(s, \rho \times \pi)
\end{align}
for any irreducible admissible representation $\rho$ of $GL$. 

The Langlands classification implies that $\pi'$ can be considered to be the unique quotient of the following induced representation:
$$\nu^{r_1} \tau_1' \times \cdots \times \nu^{r_m} \tau_m' \rtimes\pi_t'$$
where $\tau_i'$ is a tempered representation of $GL$ for each $i=1, \cdots, m$ with $r_1 > r_2 > \cdots > r_m>0$, and $\pi_t'$ is a tempered representation of $G_{n'}'$ ($G_{n'}'$ and $G_n$ are the same type of groups). Since $\pi'$ is non-tempered, $m \geq 1$.

Langlands classification (Assumption \ref{Assumption:LS}) implies that for any admissible representation $\rho$ we have
$$L(s, \rho \times \pi) = L(s, \rho \times \pi') = L(s, \rho \times \pi_t') \displaystyle\prod\limits_{i=1}^{m} L(s+r_i, \rho \times \tau_i')L(s-r_i, \rho \times \widetilde{\tau_i'} \otimes (\omega_{\pi'} \circ det)).$$
If we let $\rho$ be contragredient of $\widetilde{\tau_i'} \otimes (\omega_{\pi'} \circ det)$, then $L(s - r_j, \rho \times \widetilde{\tau_i'} \otimes (\omega_{\pi'} \circ det))$ has a pole at $s= r_j>0$. This also implies that $L(s, \rho \times \pi)=L(s, \rho \times \pi')$ also has a pole at $s=r_j>0$ when  $\rho$ is contragredient of $\widetilde{\tau_i'} \otimes (\omega_{\pi'} \circ det)$. This contradicts that $L(s, \rho \times \pi)$ is holomorphic for $Re(s) >0$ since both $\rho$ and $\pi$ are tempered and generic (the tempered $L$-function conjecture \cite{Heiermann_Opdam:2013}). Therefore, $m=0$. In other words, $\pi'$ is a tempered representation as well. We conclude that all members in $\Pi_{\phi}$ are tempered.
\end{proof}

\subsection{A strong version of the generic Arthur packet conjecture}
\label{Sec5.3}
We first prove one main lemma that proves that temperedness is preserved through the local Langlands functoriality from $GSpin$ groups to $GL$ groups. Recall that we, in Proposition \ref{tempered}, showed that the local functorial lift of a tempered generic representation of $GSpin$ groups is a tempered representation of $GL$ as well. The next natural question is whether a non-tempered representation of $GSpin$ can be lifted to a tempered representation of $GL$. The following lemma answers this question in the case of $GSpin$ groups:

\begin{lem}\label{TT}
Let $\pi$ be an irreducible admissible generic representation of $GSpin$. Then, $\pi$ is a tempered representation of $GSpin$ if and only if its local functorial lift $\Pi$ is a tempered representation of $GL$. 
\end{lem}

\begin{proof}
Proposition \ref{tempered} implies the sufficient condition. Now, suppose that $\pi$ is a non-tempered admissible generic representation of $G_n$ and its local functorial lift $\Pi$ is a tempered representation of $GL$. Then due to \cite{Heiermann_Muic:2007}, there exists $r_1 > r_2 > \cdots > r_m>0$ such that $\pi$ can be written as the following induced representation:
$$\nu^{r_1} \tau_1 \times \cdots \times \nu^{r_m} \tau_m \rtimes \pi_t$$
where $\tau_i$ is a tempered representation of $GL$ for each $i=1, \cdots, m$ and $\pi_t$ is a tempered generic representation of $G_{n'}$ for some $n' < n$. Since $\pi$ is non-tempered, $m \geq 1$.

Now, we consider the $L$-function $L(s, \widetilde{\tau_j} \times \pi)$. Since the $L$-functions from Langlands-Shahidi method are defined by the Langlands classification, for any admissible representation $\rho$ we have
$$L(s, \rho \times \pi) = L(s, \rho \times \pi_t) \displaystyle\prod\limits_{i=1}^{m} L(s+r_i, \rho \times \tau_i)L(s-r_i, \rho \times \widetilde{\tau_i} \otimes (\omega_{\pi} \circ det)).$$
As in the proof of Lemma \ref{tempered $L$-packet}, we conclude that $L(s, \rho \times \pi)$ has a pole at $s=r_j>0$ when $\rho$ is contragredient of $\widetilde{\tau_i} \otimes (\omega_{\pi} \circ det)$ and this contradicts the tempered $L$-function conjecture \cite{Heiermann_Opdam:2013}. Therefore, $\pi$ is also a tempered representation. 
\end{proof}

Let us now introduce the strong version of the generic Arthur packet conjecture.

\begin{conjecture}\label{Strong G}
Let $\psi, \phi_{\psi},$ be as in Theorem \ref{ConjG}. Let $\Pi(\phi_{\psi})$ be an $L$-packet that corresponds to $\phi_{\psi}$ that is defined in Definition \ref{Def:generic L-packet}, which means we also assume that $\Pi(\phi_{\psi})$ contains a generic member. Then the $L$-packet $\Pi(\phi_{\psi})$ is a tempered $L$-packet.
\end{conjecture}

Now, we are ready to prove a strong version of the generic Arthur packet conjecture in the case of $GSpin$ groups.

\begin{thrm}\label{StrongConjG}
Let $\textbf{G}$ be a $GSpin$ groups and let $\psi \in \Psi(G)$ be an Arthur parameter. Suppose the $L$-packet $\Pi(\phi_{\psi})$ that corresponds to the $L$-parameter $\phi_{\psi}$ has a generic member, then $\Pi(\phi_{\psi})$ is a tempered $L$-packet.
\end{thrm}

\begin{proof}
Let $\pi$ be a generic member in $\Pi(\phi_{\psi})$. 
Theorem \ref{ConjG:GSpin} implies that the corresponding $L$-parameter $\phi_{\psi}: L_F \rightarrow \,^{L}G_n = \hat{G}_n \rtimes L_F$ is a tempered $L$-parameter. Therefore, an image $\phi_{\psi}(W_F)$ in $\hat{G}_n$ is bounded. Since $\hat{G}_n$ is either $GSp_{2n}(\mathbb{C})$ or $GSO_{2n}(\mathbb{C})$, we have a natural embedding 
$$\iota: \,^{L}G_n \hookrightarrow \,^{L}GL_{2n} = GL_{2n}(\mathbb{C}) \rtimes L_F.$$ Then, an image $\iota \circ \phi_{\psi}(W_F)$ in $\hat{GL}_{2n}$ is also bounded. Therefore, $\iota \circ \phi_{\psi}$ becomes a tempered $L$-parameter for $GL_{2n}$. By the local Langlands correspondence for general linear groups \cite{Harris_Taylor:2001, Henniart:2000}, there exists an irreducible admissible representation $\Pi'$ that corresponds to $\iota \circ \phi_{\psi}$ such that
\begin{align}\label{E:GL}
\gamma(s, \rho \times \Pi', \psi_F) = \gamma(s, \phi_{\rho} \otimes (\iota \circ \phi_{\psi}), \psi_F) \textit{\ \ and \ \ } L(s, \rho \times \Pi') = L(s, \phi_{\rho} \otimes (\iota \circ \phi_{\psi}))
\end{align}
for any irreducible admissible representation $\rho$ of $GL$, where $\phi_{\rho}$ is the $L$-parameter that corresponds to $\rho$ through the local Langlands correspondence for general linear groups \cite{Harris_Taylor:2001, Henniart:2000}. Here, the local factors in the left hand side are Rankin-Selberg local factors \cite{Jacquet_PS_Shalika:1983} and the local factors in the right hand side are Artin factors. The equalities \eqref{E:GL} also imply the following:
$$\gamma(s, \rho \times \pi, \psi_F) = \gamma(s, \phi_{\rho} \otimes \phi_{\psi}, \psi_F) = \gamma(s, \phi_{\rho} \otimes (\iota \circ \phi_{\psi}), \psi_F)= \gamma(s, \rho \times \Pi', \psi_F)$$
and
$$L(s, \rho \times \pi) = L(s, \phi_{\rho} \otimes \phi_{\psi}) = L(s, \phi_{\rho} \otimes (\iota \circ \phi_{\psi}))= L(s, \rho \times \Pi').$$
Therefore, $\Pi'$ is a local functorial lift of $\pi$.

In the case of general linear groups, it is known that an irreducible admissible representation that corresponds to a tempered $L$-parameter is a tempered representation (\cite{Harris_Taylor:2001, Henniart:2000, Langlands:1989} or Theorem 1.3.1 in \cite{Arthur:2013}). Therefore, $\Pi'$ is a tempered representation of $GL$ since $\iota \circ \phi_{\psi}$ is a tempered $L$-parameter for $GL$. In other words, we conclude that the local functorial lift $\Pi'$ of $\pi$ is a tempered representation of $GL$. Then, Lemma \ref{TT} implies that $\pi$ is a tempered representation as well. Now, it remains to prove that all other members in $\Pi(\phi_{\psi})$ are tempered as well. Since $\Pi(\phi_{\psi})$ contains a tempered generic representation $\pi$, Lemma \ref{tempered $L$-packet} implies that $\Pi(\phi_{\psi})$ is a tempered $L$-packet.
\end{proof}

\begin{remark}
Non-triviality of $\psi$ to the second $SL_2(\mathbb{C})$ gives non-tempered representation in an $L$-packet. More precisely, assume that $\Pi(\phi_{\psi})$ is a tempered $L$-packet. Suppose also that $\psi|SL_2(\mathbb{C})$ is non-trivial. Then, $\phi_{\psi}(W_F)|\widehat{G}$ is unbounded since $\phi_{\psi}(w^m)=\psi(w^m, \begin{pmatrix} |w|^{\frac{m}{2}} & 0 \\ 0 & |w|^{-\frac{m}{2}} \end{pmatrix})$ for any positive integer $m$. This also implies that the corresponding Langlands parameter for general linear groups through local functorial lift is also unbounded. Therefore, the functorial lift is non-tempered representation of general linear groups. Since tempered representation of $GSpin$ can not be lifted to non-tempered representation of $GL$, $\Pi(\phi_{\psi})$ contains non-tempered representation, which contradicts the assumption. Therefore, $\psi|SL_2(\mathbb{C})$ is trivial. We conclude that $\psi|L_F=\phi_{\psi}$ is also tempered.

\end{remark}

\begin{remark}\label{Classical groups}
The techniques used in this section can be applied to the cases of classical groups and will produce more interesting applications since we can compare our definition of $L$-packets with the works of Arthur and Mok \cite{Arthur:2013, Mok:2015} in those cases. We leave this for our future work.
\end{remark}

\proof[Acknowledgements]
This project is suggested by my advisor, F. Shahidi and the part of this paper contains the main result of my doctoral dissertation at Purdue University (Chapter \ref{Sec4}). I would like to express my deepest gratitude to F. Shahidi for his constant encouragement and help. I also thanks M. Asgari, M. Harris, and W. Li for their questions and discussion during my talks at Morningside center and at Midwest Representation Theory conference at University of Iowa (in honor of P. Kutzko's 70th birthday conference). I started to work on chapter \ref{Sec5} while I tried to answer W. Li's question and Remark \ref{Harris}(iii) is pointed out by M. Harris and discussion with M. Asgari. I also thanks D. Goldberg, V. Heiermann, M. Krishnamurthy, J. Lau, B. Liu, L. Lomel\'i, C. Mok, and S. Varma for many helpful discussions and suggestions. I am grateful to Max-Planck-Institut f\"ur Mathematik for providing excellent working conditions during a two-month visit (June 2015 - July 2015). The part of Chapter \ref{Sec5} was produced while I was visiting the Institute. This work was supported by the National Research Foundation of Korea (NRF) grant funded by the Korea government (MSIP) (No. 2017R1C1B2010081).

\end{document}